\documentclass[11pt,a4paper]{article}
\usepackage{amsmath}
\usepackage[utopia]{mathdesign}

\usepackage{cite}
\usepackage[left=2.5cm,right=2.5cm,top=2.5cm,bottom=2.5cm]{geometry}
\usepackage{graphicx,xcolor}
\usepackage{float}
\usepackage{array}
\usepackage{tikz}
\usetikzlibrary{decorations.pathreplacing}
\usepackage[unicode,hidelinks]{hyperref}
\usepackage{parskip}
\usepackage{enumerate} 
\usepackage[shortlabels]{enumitem} 
\numberwithin{equation}{section}

\def\arraystretch{1.25}	

\usepackage[thmmarks,standard,amsmath,hyperref]{ntheorem}
\theoremheaderfont{\normalfont\bfseries} 
\theorembodyfont{\normalfont}
\theoremseparator{.}

\theorembodyfont{\itshape}
\renewtheorem{theorem}{Theorem}[section]
\renewtheorem{proposition}[theorem]{Proposition}
\renewtheorem{lemma}[theorem]{Lemma}
\renewtheorem{corollary}[theorem]{Corollary}

\theorembodyfont{\normalfont}
\renewtheorem{definition}[theorem]{Definition}
\renewtheorem{remark}[theorem]{Remark}
\renewtheorem{example}[theorem]{Example}

\theoremstyle{plain}
\theoremstyle{nonumberplain}
\theoremsymbol{\ensuremath{\square}}
\renewtheorem{proof}{Proof}

\title{\bf First-Order Characterizations of Strong Quasiconvexity and a Quasimonotonicity Gap}
\author{Nguyen Xuan Duy Bao$^{1}$, Nguyen Mau Nam$^2$}
\date{}

\def\R{\mathbb{R}}

\def\la{\langle}
\def\ra{\rangle}

\begin{document}
\maketitle

\footnotetext[1]{\strut Independent Researcher, Ho Chi Minh City, Vietnam. \emph{Email}: \texttt{nxdbao@gmail.com}}
\footnotetext[2]{\strut Fariborz Maseeh Department of Mathematics and Statistics, Portland State University, Portland, OR 97207, USA. \emph{Email}: \texttt{mnn3@pdx.edu}}

\small
{\bf Abstract.} We study first-order conditions for G\^ateaux differentiable strongly quasiconvex functions on open convex subsets of normed spaces. First, we prove that the value-order gradient condition characterizes strong quasiconvexity with the same modulus $\sigma$, improving the modulus $\sigma/2$ of Vladimirov, Nesterov, and Chekanov and of Hadjisavvas and Lara. Second, we answer negatively a question of Lara, Marcavillaca, and Vuong: for every $\sigma>0$ there is a $C^\infty$ function on the real line that satisfies their generalized quasimonotonicity condition with modulus $\sigma$, even in its nonstrict form, yet is strongly quasiconvex with no positive modulus. Third, we show that adding an interior test along segments to this condition yields an exact characterization of strong quasiconvexity. Finally, a critical-point condition on the same compensated profile gives a second, elementary characterization.


\normalsize

\section{Introduction and Preliminaries}
{Quasiconvexity is a classical notion in optimization and economics. A function is quasiconvex precisely when all its sublevel sets are convex. In the differentiable setting, the Arrow--Enthoven condition \cite{arrow61} characterizes quasiconvexity through the first-order implication
    \[
    f(x)\le f(y)\quad \Longrightarrow\quad \langle \nabla f(y),y-x\rangle \ge 0 ,
    \]
which links quasiconvexity to quasimonotonicity of the gradient; see also \cite{cambini09,hadjisavvas05,ponstein67}.}

Strong quasiconvexity, introduced by Polyak \cite{polyak66}, adds a uniform quadratic term to the quasiconvexity inequality. It does not imply convexity, but whenever a minimizer exists, that minimizer is unique and satisfies a quadratic growth estimate. The class has received renewed attention in optimization: existence results and proximal point algorithms \cite{lara22}, the extension to normed spaces \cite{nam25}, equilibrium methods \cite{lara24}, and accelerated dynamics \cite{hlmv26}; see \cite{grad25} for a survey.

Table~\ref{tab:conditions} collects the first-order conditions studied in this paper. There, $X$ is a real normed space, $C\subseteq X$ is a nonempty open convex set, $f\colon C\to\R$ is G\^ateaux differentiable, $\sigma\ge0$, $x,y\in C$ are distinct, and
\[
\varphi_{x,y}(t):=f(x+t(y-x))+\frac{\sigma}{2}t(1-t)\|y-x\|^2,\qquad t\in[0,1],
\]
is the compensated profile of Definition~\ref{def:profile}; $(\widehat C_\sigma)$ is used only in the counterexample.

\begin{table}[H]
\centering
\small
\renewcommand{\arraystretch}{1.5}
\begin{tabular}{@{}c >{\raggedright\arraybackslash}p{11.5cm}@{}}
\hline
Label & Condition (for all distinct $x,y\in C$) \\
\hline
$(A_\sigma)$ & $f(\lambda x+(1-\lambda)y)\le\max\{f(x),f(y)\}-\tfrac{\sigma}{2}\lambda(1-\lambda)\|x-y\|^2$ for $\lambda\in(0,1)$ \\
$(B_\sigma)$ & $f(x)\le f(y)\ \Rightarrow\ \la\nabla f(y),x-y\ra\le-\tfrac{\sigma}{2}\|x-y\|^2$ \\
$(C_\sigma)$ & $\la\nabla f(x),y-x\ra>-\tfrac{\sigma}{2}\|x-y\|^2\ \Rightarrow\ \la\nabla f(y),x-y\ra\le-\tfrac{\sigma}{2}\|x-y\|^2$ \\
$(\widehat C_\sigma)$ & $(C_\sigma)$ with premise $\ge$ \\
$(D_\sigma)$ & $\varphi_{x,y}'(s)>0>\varphi_{x,y}'(t)$, $0<s<t<1$ $\Rightarrow$ $\min\{\varphi_{x,y}(s),\varphi_{x,y}(t)\}\le\max\{f(x),f(y)\}$ \\
$(E_\sigma)$ & $\varphi_{x,y}'(t)=0$, $0<t<1$ $\Rightarrow$ $\varphi_{x,y}(t)\le\max\{f(x),f(y)\}$ \\
\hline
\end{tabular}
\caption{First-order conditions studied in this paper.}
\label{tab:conditions}
\end{table}

Two first-order analogues of the Arrow--Enthoven condition have been proposed for strong quasiconvexity, and the status of both was unsettled. The first is the \emph{value-order gradient condition} $(B_\sigma)$ of Table~\ref{tab:conditions}. That $\sigma$-quasiconvexity implies $(B_\sigma)$ is classical \cite{vladimirov78}; see \cite[Theorem~3]{lara25}. For the converse, the original argument of \cite{vladimirov78} yields modulus $\sigma/2$ and rests on a lemma whose proof contains an error, as pointed out by Hadjisavvas and Lara, who gave a self-contained proof that $(B_\sigma)$ implies $(\sigma/2)$-quasiconvexity \cite[Section~1 and Corollary~8]{hadjisavvaslara25}. Whether $(B_\sigma)$ recovers the full modulus $\sigma$ was not addressed. The second is the \emph{generalized quasimonotonicity condition} $(C_\sigma)$ introduced by Lara, Marcavillaca, and Vuong: strong quasiconvexity implies it \cite[Proposition~8]{lara25}, the converse was explicitly left open with no counterexample known \cite[Remark~9]{lara25}, and on open convex subsets of $\R^n$ a stronger nonstrict form $(\widehat C_\sigma)$ was shown to be necessary as well \cite[Proposition~10]{lara25}. The reverse implications are stated there as open; the first arXiv version moreover claimed that $(A_\sigma)$, $(\widehat C_\sigma)$, and $(C_\sigma)$ are equivalent \cite[Remark~12]{lara24preprint}.

The contributions of this paper are the following.
\begin{enumerate}[(1)]
\item \emph{Exact modulus.} Under G\^ateaux differentiability, $(B_\sigma)$ characterizes $\sigma$-quasiconvexity with the same modulus (Theorem~\ref{thm:sqc}), which improves the modulus $\sigma/2$ of \cite{vladimirov78} and \cite[Corollary~8]{hadjisavvaslara25}. The proof follows both endpoints of each one-dimensional sublevel interval, rather than the segment from one endpoint, and this is what recovers $\sigma$ instead of $\sigma/2$.
\item \emph{Negative answer.} For every $\sigma>0$ there is $f_\sigma\in C^\infty(\R)$ satisfying $(\widehat C_\sigma)$, hence $(C_\sigma)$, which is $\mu$-quasiconvex for no $\mu>0$ (Example~\ref{ex:no-positive-modulus} and Remark~\ref{rem:smooth-counterexample}). This answers the question of \cite[Remark~9]{lara25} and refutes the claim of \cite[Remark~12]{lara24preprint}. In terms of the two moduli of Proposition~\ref{prop:moduli}, $\operatorname{qm}(f_\sigma)=\sigma$ while $\operatorname{sqc}(f_\sigma)=0$ (Remark~\ref{rem:gap-maximal}).
\item \emph{Exact characterizations.} Two conditions on the compensated profile $\varphi_{x,y}$, a sign-change condition $(D_\sigma)$ and a critical-point condition $(E_\sigma)$, characterize $\sigma$-quasiconvexity (Theorem~\ref{thm:completed-characterization}, Proposition~\ref{prop:critical-point-characterization}); for $\sigma=0$ both reduce to Theorem~\ref{thm:qc}. Moreover, $(D_\sigma)$ consists of $(C_\sigma)$ and of an interior test that $(C_\sigma)$ cannot see, which explains the gap in (2) (Remark~\ref{rem:completed-characterization}).
\end{enumerate}

In summary, for $\sigma\ge0$ (with $\sigma>0$ in the counterexample),
\[
(A_\sigma)\ \overset{\text{\scriptsize Thm.~\ref{thm:sqc}}}{\Longleftrightarrow}\ (B_\sigma)\ \overset{\text{\scriptsize Prop.~\ref{prop:necessity-C}}}{\Longrightarrow}\ (C_\sigma),
\qquad
(\widehat C_\sigma)\ \overset{\text{\scriptsize Ex.~\ref{ex:no-positive-modulus}}}{\nRightarrow}\ (A_\mu)\ \text{for any }\mu>0,
\qquad
(D_\sigma)\ \overset{\text{\scriptsize Thm.~\ref{thm:completed-characterization}}}{\Longleftrightarrow}\ (A_\sigma)\ \overset{\text{\scriptsize Prop.~\ref{prop:critical-point-characterization}}}{\Longleftrightarrow}\ (E_\sigma).
\]
{The paper is organized as follows. Section 2 recalls the classical {first-order} characterizations of quasiconvexity. Section 3 proves the exact-modulus value-order characterization and {then answers the open converse question negatively}. Section 4 gives two further exact characterizations, {both phrased through a compensated profile}. Section 5 summarizes the conclusions.}

Throughout this paper, let $X$ be a real normed space, and let $X^{*}$ denote its topological dual. We use the standard duality pairing $\langle x^{*}, x\rangle = x^{*}(x)$ for $x^{*} \in X^{*}$ and $x \in X$. {Let $C\subseteq X$ be open. We call $f\colon C\to\R$ \emph{G\^ateaux differentiable} at $x\in C$ if there is a functional $\nabla f(x)\in X^*$, its \emph{G\^ateaux derivative} at $x$, such that $\la\nabla f(x),h\ra=\lim_{t\to0}\tfrac1t\bigl(f(x+th)-f(x)\bigr)$ for every $h\in X$.}

\begin{definition}[{$\sigma$-quasiconvexity}]\label{def:sqc} Let ${C}$ be a nonempty convex subset of $X$. For $\sigma \ge 0$, a function $f: {C} \to \R$ is said to be \emph{$\sigma$-quasiconvex} if for all $x,y \in {C}$ and all $\lambda \in (0,1)$,
\begin{equation}\tag{$A_\sigma$}\label{eq:SQC}
	f(\lambda x+(1-\lambda)y) \leq \max\{f(x),f(y)\} - \dfrac{\sigma}{2}\lambda(1-\lambda)\,\|x-y\|^2.
\end{equation}

\begin{enumerate}[\rm (i)]
    \item If $\sigma=0$, then we say that $f$ is {\em quasiconvex};
    \item If $\sigma>0$, then we say that $f$ is {\em strongly quasiconvex} with modulus $\sigma$.
\end{enumerate}
\end{definition}

\section{\texorpdfstring{{Classical first-order characterizations of quasiconvexity}}{Classical first-order characterizations of quasiconvexity}}
{%
We first record the classical first-order characterizations of differentiable quasiconvexity; see \cite{arrow61,cambini09,giorgi14,hadjisavvas05,ponstein67}. Besides fixing notation, this result identifies the condition to which the strong counterpart reduces when $\sigma=0$.

\begin{theorem}\label{thm:qc}
Let $C\subseteq X$ be open and convex, and let $f\colon C\to\R$ be G\^ateaux differentiable. Then the following conditions are equivalent:
\begin{enumerate}[(a)]
\item The function $f$ is quasiconvex.

\item For all $x,y\in C$,
\[
f(x)\le f(y)
\quad\Longrightarrow\quad
\la\nabla f(y),y-x\ra\ge0.
\]

\item For all $x,y\in C$,
\[
\la\nabla f(x),y-x\ra>0
\quad\Longrightarrow\quad
\la\nabla f(y),y-x\ra\ge0.
\]
\end{enumerate}
\end{theorem}

\begin{proof}
\emph{{\rm (a)} $\Rightarrow$ {\rm (b)}.}
Assume that $f(x)\le f(y)$. For $0<t<1$, quasiconvexity gives
\[
f(y+t(x-y))\le f(y).
\]
Subtracting $f(y)$, dividing by $t$, and letting $t\downarrow0$ yields
\[
\la\nabla f(y),x-y\ra\le0,
\]
which is equivalent to {\rm (b)}.

\emph{{\rm (b)} $\Rightarrow$ {\rm (c)}.}
Suppose that $\la\nabla f(x),y-x\ra>0$. If $f(y)\le f(x)$, condition~{\rm (b)} applied to $(y,x)$ gives
\[
\la\nabla f(x),y-x\ra\le0,
\]
a contradiction. Hence $f(x)<f(y)$, and a second application of {\rm (b)}, now to $(x,y)$, yields {\rm (c)}.

\emph{{\rm (c)} $\Rightarrow$ {\rm (a)}.}
Suppose that $f$ is not quasiconvex. Then there exist distinct $x,y\in C$ and $r\in(0,1)$ such that, for
\[
g(t):=f(x+t(y-x)),\qquad t\in[0,1],
\]
we have $g(r)>\max\{g(0),g(1)\}$. The mean value theorem gives $s\in(0,r)$ and $t\in(r,1)$ such that
\[
g'(s)>0>g'(t).
\]
Set $x_s=x+s(y-x)$ and $x_t=x+t(y-x)$. Since $t-s>0$,
\[
\la\nabla f(x_s),x_t-x_s\ra=(t-s)g'(s)>0,
\]
whereas
\[
\la\nabla f(x_t),x_t-x_s\ra=(t-s)g'(t)<0.
\]
This contradicts {\rm (c)}. Therefore $f$ is quasiconvex.
\end{proof}
}%
\section{\texorpdfstring{{Gradient conditions and the quasimonotonicity gap}}{Gradient conditions and the quasimonotonicity gap}}

\subsection{\texorpdfstring{{The exact-modulus value-order characterization}}{The exact-modulus value-order characterization}}
{%
We first prove that the value-order gradient condition characterizes strong quasiconvexity with the same modulus.

\begin{lemma}[One-dimensional exact-modulus lemma]\label{lem:one-dimensional-exact}
Let $g$ be differentiable on an open interval containing $[0,1]$, and let $c>0$. Suppose that
\begin{equation}\label{eq:one-dimensional-B}
g(u)\le g(v)
\quad\Longrightarrow\quad
g'(v)(v-u)\ge c\,(v-u)^2,
\qquad u,v\in[0,1].
\end{equation}
Then
\begin{equation}\label{eq:goal-one-dimensional}
g(t)\le\max\{g(0),g(1)\}-c\,t(1-t),
\qquad t\in(0,1).
\end{equation}
\end{lemma}

\begin{proof}

Condition~\eqref{eq:one-dimensional-B} makes $g$ a valley, so each sublevel set is an interval $[\ell(\lambda),r(\lambda)]$ (Steps~1--2). As $\lambda$ rises, its endpoints separate at speeds bounded by $1/(c\,w(\lambda))$ (Step~3); integrating the two bounds separately through the potential $\Phi$ yields the quadratic deficit (Steps~4--5).

\emph{Step 1 (the valley shape).}
We will show that the function $g$ is strictly quasiconvex on $[0,1]$. Here this means that $g(s)<\max\{g(u),g(v)\}$ whenever $u<s<v$. Suppose on the contrary that $g$ is not strictly quasiconvex on $[0,1]$. Then there exist $u<s<v$ in $[0,1]$ with $g(s)\ge\max\{g(u),g(v)\}$. Therefore, \eqref{eq:one-dimensional-B} applied to $(u,s)$ and $(v,s)$ gives $g'(s)\ge c(s-u)>0$ and $g'(s)\le-c(v-s)<0$, which is a contradiction. A continuous strictly quasiconvex function on $[0,1]$ decreases strictly to a unique minimizer $t_0$, then increases strictly \cite[Corollary~2.5.2]{cambini09}. Either branch may reduce to a point (Figure~\ref{fig:valley}). Write
\[
m:=g(t_0)=\min_{s\in[0,1]}g(s),\qquad M:=\max\{g(0),g(1)\}.
\]

\emph{Step 2 (the sublevel interval and its endpoints).}
The functions $\ell(\lambda)$ and $r(\lambda)$ are the two ends of the sublevel interval at height $\lambda$, and $w(\lambda)$ is its width.
For every $\lambda\in[m,M]$ the sublevel set is the interval cut from the valley at height $\lambda$ (Figure~\ref{fig:valley}),
\[
\{s\in[0,1]:g(s)\le\lambda\}=[\ell(\lambda),r(\lambda)],\qquad w(\lambda):=r(\lambda)-\ell(\lambda),
\]
whose endpoints trace the two branches of $g$:
\[
\ell=\bigl(g|_{[0,t_0]}\bigr)^{-1}\ \text{ on }[m,g(0)],\qquad \ell\equiv0\ \text{ on }[g(0),M],
\]
\[
r=\bigl(g|_{[t_0,1]}\bigr)^{-1}\ \text{ on }[m,g(1)],\qquad r\equiv1\ \text{ on }[g(1),M].
\]
Hence $\ell$ is continuous and nonincreasing, $r$ is continuous and nondecreasing, $w$ is nondecreasing, and $[\ell(M),r(M)]=[0,1]$.

Fix $t\in(0,1)$. Set $\eta:=g(t)$ and $[a,b]:=[\ell(\eta),r(\eta)]$. Then $\eta<M$, since strict quasiconvexity gives $g(t)<\max\{g(0),g(1)\}=M$. Moreover, $g<\eta$ on the open interval $(a,b)$ by the strict monotonicity of the two branches of $g$. The equality $g(t)=\eta$ therefore makes $t$ an endpoint, $t\in\{a,b\}$ (Figure~\ref{fig:valley}). Suppose first that $\eta>m$. {Then $t_0\in[a,b]$, and $w_0:=b-a>0$ because $g(t_0)=m<\eta$.} The boundary case $\eta=m$ is treated in Step 6.

\begin{figure}[ht]
\centering
\begin{tikzpicture}[>=stealth,line join=round,
     dot/.style={circle,fill,inner sep=1.2pt},
     ref/.style={dashed,gray}]
  \draw[ultra thick,blue!55!black]
     plot[domain=0:0.4,samples=40] ({10*\x},{0.9+22.5*((0.4-\x)*(0.4-\x))});
  \draw[ultra thick,blue!55!black]
     plot[domain=0.4:1,samples=60] ({10*\x},{0.9+10.8*((\x-0.4)*(\x-0.4))});
  \node[blue!55!black] at (9.85,4.1){$g$};
  \draw[->] (0,0) -- (0,5.3);
  \draw[->] (-0.3,0) -- (10.9,0);
  \draw[ref] (0,4.788) -- (10,4.788);
  \draw[ref] (0,3.15)  -- (8.56,3.15);
  \draw[ref] (0,1.872) -- (7.0,1.872);
  \draw[ref] (0,0.9)   -- (4,0.9);
  \draw[ref] (0.838,0) -- (0.838,3.15);
  \draw[ref] (1.92,0)  -- (1.92,1.872);
  \draw[ref] (7.0,0)   -- (7.0,1.872);
  \draw[ref] (8.56,0)  -- (8.56,3.15);
  \draw[ref] (4,0)     -- (4,0.9);
  \draw[ref] (10,0)    -- (10,4.788);
  \draw[very thick,orange!85!black] (0.838,0) -- (8.56,0);
  \node[dot] at (4,0.9){};
  \node[dot] at (0.838,3.15){};
  \node[dot] at (8.56,3.15){};
  \node[dot] at (1.92,1.872){};
  \node[dot] at (7.0,1.872){};
  \foreach \xx/\lab in {0/$0$, 0.838/$\ell(\lambda)$, 1.92/$a$, 4/$t_0$, 7.0/$b$, 8.56/$r(\lambda)$, 10/$1$}{%
     \draw (\xx,0.1) -- (\xx,-0.1);
     \node[below] at (\xx,-0.12){\lab};}
  \foreach \yy/\lab in {0.9/$m$, 1.872/$\eta$, 3.15/$\lambda$, 4.788/$M$}{%
     \draw (-0.1,\yy) -- (0.1,\yy);
     \node[left] at (-0.15,\yy){\lab};}
  \draw[decorate,decoration={brace,amplitude=4pt,mirror}] (0.838,-0.7) -- (1.92,-0.7) node[midway,below=2pt]{$U$};
  \draw[decorate,decoration={brace,amplitude=4pt,mirror}] (1.92,-0.7)  -- (7.0,-0.7)  node[midway,below=2pt]{$w_0$};
  \draw[decorate,decoration={brace,amplitude=4pt,mirror}] (7.0,-0.7)   -- (8.56,-0.7) node[midway,below=2pt]{$V$};
  \draw[decorate,decoration={brace,amplitude=4pt,mirror}] (0.838,-1.55) -- (8.56,-1.55) node[midway,below=2pt]{$w$};
\end{tikzpicture}
\caption{The function $g$ has a valley shape, and its sublevel set $\{g\le\lambda\}=[\ell(\lambda),r(\lambda)]$ (orange) widens as $\lambda$ increases from $m$ to $M$.}
\label{fig:valley}
\end{figure}

\emph{Step 3 (how fast the endpoints move).}
Let $\lambda\in[\eta,M)$, so that $w(\lambda)\ge w_0>0$. We bound the right endpoint first. Suppose that $r(\lambda)<1$, that is $\lambda<g(1)$. Then $g(\ell(\lambda))\le\lambda=g(r(\lambda))$. Applying \eqref{eq:one-dimensional-B} to $(\ell(\lambda),r(\lambda))$ and dividing by $w(\lambda)$ gives
\begin{equation}\label{eq:branch-derivative}
g'(r(\lambda))\ge c\,w(\lambda)>0 .
\end{equation}
Since $g$ is continuous and strictly increasing on $[t_0,1]$ and is differentiable at $r(\lambda)$, with $g'(r(\lambda))\ne0$ by \eqref{eq:branch-derivative}, its inverse $r$ is differentiable at $\lambda$, with $r'(\lambda)=1/g'(r(\lambda))$. With \eqref{eq:branch-derivative} this yields $0\le r'(\lambda)\le 1/(c\,w(\lambda))$. For $\lambda>g(1)$, on the other hand, $r\equiv1$ by Step~2, so $r'(\lambda)=0$. Thus $r$ is differentiable on $[\eta,M]$ (one-sidedly at its endpoints) except possibly at $g(1)$.

The left endpoint is symmetric. When $\ell(\lambda)>0$, that is $\lambda<g(0)$, the same steps give $-g'(\ell(\lambda))\ge c\,w(\lambda)>0$ and $0\le-\ell'(\lambda)\le 1/(c\,w(\lambda))$, while $\ell\equiv0$ for $\lambda>g(0)$ by Step~2, so $\ell$ is differentiable except possibly at $g(0)$. Collecting the two bounds,
\begin{equation}\label{eq:level-endpoint-speeds}
0\le-\ell'(\lambda)\le\frac{1}{c\,w(\lambda)},\qquad
0\le r'(\lambda)\le\frac{1}{c\,w(\lambda)}
\qquad{\text{for }\lambda\in[\eta,M]\setminus\{g(0),g(1)\}.}
\end{equation}
As $w\ge w_0$, these derivatives never exceed $1/(c\,w_0)$ in absolute value. Since both $\ell$ and $r$ are continuous on $[\eta,M]$ and differentiable off the two levels $g(0),g(1)$, the mean value theorem gives the same Lipschitz bound on every member of the resulting finite partition. Continuity at the partition points then makes both maps globally Lipschitz, hence absolutely continuous, on $[\eta,M]$.

\emph{Step 4 (a potential in the level variable).}
Relative to the fixed interval $[a,b]$ at level $\eta$, the functions $U$ and $V$ measure how far the left and right endpoints have travelled outward from $a$ and $b$, respectively.
Let $U:=a-\ell$ and $V:=r-b$. Both are nonnegative, nondecreasing, and absolutely continuous, and
\[
w=w_0+U+V,\qquad
U(\eta)=V(\eta)=0,\qquad U(M)=a,\qquad V(M)=1-b.
\]
By \eqref{eq:level-endpoint-speeds},
\begin{equation}\label{eq:level-maximum-speed}
c\,w\max\{U',V'\}\le1
\qquad\text{a.e. on }[\eta,M].
\end{equation}
Define
\[
\Phi(U,V):=w_0\max\{U,V\}+\frac{U^2+V^2}{2}.
\]
The identity $\max\{U,V\}=(U+V+|U-V|)/2$ shows that $\max\{U,V\}$ is absolutely continuous. Its derivative equals $U'$ on $\{U>V\}$ and $V'$ on $\{U<V\}$; on $\{U=V\}$ we have $U'=V'$ almost everywhere, since at every accumulation point of this set at which both functions are differentiable, $(U-V)'=0$, and the isolated points form a countable set. Hence
\[
\frac{d}{d\lambda}\max\{U,V\}\le\max\{U',V'\}
\qquad\text{a.e. on }[\eta,M].
\]
Since $U,V\ge0$, we have $UU'\le U\max\{U',V'\}$ and $VV'\le V\max\{U',V'\}$. With $w=w_0+U+V$ this yields
\[
\frac{d}{d\lambda}\Phi(U,V)
\le w_0\max\{U',V'\}+UU'+VV'
\le w\max\{U',V'\}
\le\frac1c
\qquad\text{a.e. on }[\eta,M],
\]
the last inequality being \eqref{eq:level-maximum-speed}. The map $\lambda\mapsto\Phi(U(\lambda),V(\lambda))$ is absolutely continuous, {and it equals $\Phi(0,0)=0$ at $\lambda=\eta$ and $\Phi(a,1-b)$ at $\lambda=M$}. By the fundamental theorem of calculus,
\[
\Phi(a,1-b)
=\int_\eta^M\frac{d}{d\lambda}\Phi(U,V)\,d\lambda
\le\int_\eta^M\frac1c\,d\lambda
=\frac{M-\eta}{c}.
\]
Hence
\begin{equation}\label{eq:level-potential-bound}
M-\eta\ge c\,\Phi(a,1-b).
\end{equation}

\emph{Step 5 (conclusion).} We claim that
\begin{equation}\label{eq:phi-identity}
\Phi(a,1-b)=\max\{a(1-a),\,b(1-b)\}+\tfrac12(1-a-b)^2 .
\end{equation}
Suppose first that $a+b\le1$. Then $1-b\ge a$, so $\max\{a,1-b\}=1-b$ and
\[
\Phi(a,1-b)
=(b-a)(1-b)+\frac{a^2+(1-b)^2}{2}
=b(1-b)+\frac{(1-a-b)^2}{2}.
\]
Moreover $b(1-b)-a(1-a)=(b-a)(1-a-b)\ge0$, so $\max\{a(1-a),b(1-b)\}=b(1-b)$ and \eqref{eq:phi-identity} follows. If $a+b\ge1$, then $a\ge1-b$, so $\max\{a,1-b\}=a$ and
\[
\Phi(a,1-b)
=(b-a)a+\frac{a^2+(1-b)^2}{2}
=a(1-a)+\frac{(a+b-1)^2}{2},
\]
while $a(1-a)-b(1-b)=(b-a)(a+b-1)\ge0$. Again \eqref{eq:phi-identity} follows.

Since $t\in\{a,b\}$ by Step 2, we have $t(1-t)\le\max\{a(1-a),b(1-b)\}$. Hence \eqref{eq:level-potential-bound} and \eqref{eq:phi-identity} give $M-\eta\ge c\,t(1-t)$, which is \eqref{eq:goal-one-dimensional}.

\emph{Step 6 (the level of the minimum).} It remains to treat $\eta=m$, that is, $t=t_0$. Since $t<1$ and $t_0$ is the unique minimizer, $g(1)>m$. Fix $0<\varepsilon<g(1)-m$ and write
\[
\{s\in[0,1]:g(s)\le m+\varepsilon\}=[a_\varepsilon,b_\varepsilon].
\]
Then $t<b_\varepsilon<1$ and $g(b_\varepsilon)>m$, so Steps 2--5 apply at $b_\varepsilon$ and give
\[
g(b_\varepsilon)\le M-c\,b_\varepsilon(1-b_\varepsilon).
\]
As $\varepsilon\downarrow0$, the nested intervals $[a_\varepsilon,b_\varepsilon]$ shrink to $\{s\in[0,1]:g(s)\le m\}=\{t\}$, so $b_\varepsilon\downarrow t$, and passing to the limit in the last display yields \eqref{eq:goal-one-dimensional} at $t$.

This proves \eqref{eq:goal-one-dimensional}.
\end{proof}

\begin{theorem}\label{thm:sqc}
Let $\sigma\ge0$, let $C\subseteq X$ be a nonempty open convex set, and let $f\colon C\to\R$ be a G\^ateaux differentiable function. Then the following conditions are equivalent:
\begin{enumerate}[(a)]
\item The function $f$ is $\sigma$-quasiconvex.

\item For any $x,y\in C$, we have the implication:
\begin{equation*}\tag{$B_\sigma$}\label{cond:B}
f(x)\le f(y)
\quad\Longrightarrow\quad
\la\nabla f(y),x-y\ra
\le-\frac{\sigma}{2}\|x-y\|^2.
\end{equation*}
\end{enumerate}
\end{theorem}

\begin{proof}
\emph{{\rm (a)} $\Rightarrow$ {\rm (b)}.}
Assume that $f$ is $\sigma$-quasiconvex and that $f(x)\le f(y)$. For $0<t<1$, inequality~\eqref{eq:SQC} gives
\[
f(y+t(x-y))
\le f(y)-\frac{\sigma}{2}t(1-t)\|x-y\|^2.
\]
Subtracting $f(y)$, dividing by $t$, and letting $t\downarrow0$ yields $(B_\sigma)$.

\emph{{\rm (b)} $\Rightarrow$ {\rm (a)}.}
For $\sigma=0$ this is Theorem~\ref{thm:qc}, so assume $\sigma>0$. Fix distinct $x,y\in C$ and set
\[
g(s):=f\bigl(x+s(y-x)\bigr),\qquad
c:=\frac{\sigma}{2}\|y-x\|^2>0.
\]
Since $C$ is open, $g$ is differentiable on an open interval containing $[0,1]$. For $u,v\in[0,1]$ put $x_u:=x+u(y-x)$. Since $\la\nabla f(x_v),x_u-x_v\ra=(u-v)\,g'(v)$ and $\|x_u-x_v\|=|u-v|\,\|y-x\|$, condition $(B_\sigma)$ at the pair $(x_u,x_v)$ is exactly \eqref{eq:one-dimensional-B}. Lemma~\ref{lem:one-dimensional-exact} therefore yields
\[
f(x+t(y-x))
\le\max\{f(x),f(y)\}
-\frac{\sigma}{2}t(1-t)\|y-x\|^2,
\qquad t\in(0,1).
\]
Thus $f$ is $\sigma$-quasiconvex.
\end{proof}

{%
\begin{remark}\label{rem:comparison-hadjisavvas-lara}
{%
Hadjisavvas and Lara characterize the modulus in the nonsmooth setting through a value condition and its Dini-derivative form \cite[Theorems~4 and~6]{hadjisavvaslara25}. For a G\^ateaux differentiable function this specializes to \cite[Corollary~8]{hadjisavvaslara25}, where $(B_\sigma)$ is shown to imply strong quasiconvexity with modulus $\sigma/2$. Their argument first establishes strict quasiconvexity, then uses Dini derivatives and a theorem of Saks to reach the one-sided estimate
\[
f(x)\le f(x+t(y-x))
\quad\Longrightarrow\quad
f(y)-f(x+t(y-x))\ge\frac{\sigma}{4}(1-t^2)\|y-x\|^2 ,
\]
and finally invokes an equal-value point on the opposite side of the segment.

The two proofs differ in what they follow. Theirs moves along the segment from one endpoint; the proof of Lemma~\ref{lem:one-dimensional-exact} moves through the levels, following both endpoints of the sublevel interval $[\ell(\lambda),r(\lambda)]$ and bounding their speeds one at a time in~\eqref{eq:level-maximum-speed}. Used only through the sum $U'+V'$, that bound returns the modulus $\sigma/2$; used per endpoint, through the potential $\Phi$, it returns the full modulus $\sigma$. Indeed, adding the two bounds in \eqref{eq:level-endpoint-speeds} gives $w\,w'\le \tfrac{2}{c}$, hence $(w^2)'\le \tfrac{4}{c}$ a.e.\ on $[\eta,M]$ and $M-\eta\ge\tfrac{c}{4}\,(1-w_0^2)$. Since $1-(b-a)^2\ge 2\,t(1-t)$ for $t\in\{a,b\}$, this yields only $M-\eta\ge\tfrac{c}{2}\,t(1-t)$. In the setting of \cite[Corollary~8]{hadjisavvaslara25}, Theorem~\ref{thm:sqc} therefore gives modulus $\sigma$ rather than $\sigma/2$.
}%
\end{remark}
}%

\subsection{\texorpdfstring{{The quasimonotonicity gap}}{The quasimonotonicity gap}}\label{subsec:gap}

Lara, Marcavillaca, and Vuong proposed the implication $(C_\sigma)$ below as a generalized quasimonotonicity property of the gradient. They proved that strong quasiconvexity implies it on a convex subset of $\R^n$ and left the converse open \cite[Proposition~8 and Remark~9]{lara25}. We first record the necessity part in an arbitrary normed space; it is a direct consequence of Theorem~\ref{thm:sqc}.

\begin{proposition}[Necessity of generalized quasimonotonicity]\label{prop:necessity-C}
Let $\sigma\ge0$, let $C\subseteq X$ be a nonempty open convex set, and let $f\colon C\to\R$ be G\^ateaux differentiable. If $f$ is $\sigma$-quasiconvex, then for any $x,y\in C$ we have the implication:
\begin{equation*}\tag{$C_\sigma$}\label{cond:C}
\la\nabla f(x),y-x\ra
>-\frac{\sigma}{2}\|x-y\|^2
\quad\Longrightarrow\quad
\la\nabla f(y),x-y\ra
\le-\frac{\sigma}{2}\|x-y\|^2.
\end{equation*}
\end{proposition}

\begin{proof}
By Theorem~\ref{thm:sqc}, $f$ satisfies $(B_\sigma)$. Suppose that the premise of $(C_\sigma)$ holds. If $f(y)\le f(x)$, then $(B_\sigma)$ applied to the pair $(y,x)$ gives $\la\nabla f(x),y-x\ra\le-\tfrac{\sigma}{2}\|x-y\|^2$, contradicting the premise. Hence $f(x)<f(y)$, and $(B_\sigma)$ applied to $(x,y)$ gives the conclusion of $(C_\sigma)$.
\end{proof}

We now answer the converse question negatively. Write $(\widehat C_\sigma)$ for the variant of $(C_\sigma)$ in which the strict inequality in the premise is replaced by a nonstrict one; on open convex subsets of $\R^n$ it is necessary for $\sigma$-quasiconvexity when $\sigma>0$, see \cite[Proposition~10]{lara25}. The following example shows that even $(\widehat C_\sigma)$ is not sufficient for strong quasiconvexity, with any modulus.

\begin{example}\label{ex:no-positive-modulus}
{%
Let $H\colon\R\to\R$ be given by
\[
H(t):=
\begin{cases}
0,&t\le1,\\[0.5mm]
6n^4(t-n)(n+1-t),&n\le t\le n+1,\ n=1,2,\dots
\end{cases}
\]
Then $H$ is continuous, and
\begin{equation}\label{eq:H-properties}
H\ge0,
\qquad
H(x)=0\ \text{ for every }x\le1,
\qquad
H(n)=0\ \text{ for every integer }n,
\qquad
\int_n^{n+1}\!H=n^4\ \ (n\ge1).
\end{equation}
Put
\[
h(x):=x^2+\int_0^{x}\!H(t)\,dt .
\]
Then $h\in C^1(\R)$ with $h'(x)=2x+H(x)$, and $h(x)=x^2$ for $x\le0$. Since $h'\le0$ on $(-\infty,0]$ and $h'\ge0$ on $[0,\infty)$, the function $h$ is quasiconvex.

\emph{Step 1 ($h$ satisfies $(\widehat C_2)$).}
On the line, $(\widehat C_\sigma)$ reads
\begin{equation*}\tag{$\widehat C_\sigma$}\label{eq:Ghat}
f'(x)(y-x)\ge-\tfrac{\sigma}{2}|x-y|^2
\quad\Longrightarrow\quad
f'(y)(x-y)\le-\tfrac{\sigma}{2}|x-y|^2,
\qquad x,y\in\R.
\end{equation*}
We show that $h$ satisfies \eqref{eq:Ghat} with $\sigma=2$, hence also $(C_2)$. Fix $x\ne y$, write $\{x,y\}=\{s,t\}$ with $s<t$, and put $p:=s+t$. By \eqref{eq:H-properties} and $s<t$, we have
{%
\begin{equation}\label{eq:H-sign}
H(s)>0
\quad\Longrightarrow\quad
0<s<t
\quad\Longrightarrow\quad
p>0 .
\end{equation}
}%
Every inequality in ($\widehat{C}_2$) is then a statement about $p$. For instance, dividing $h'(s)(t-s)\ge-(t-s)^2$ by $t-s>0$ gives $2s+H(s)\ge s-t$, that is, $p+H(s)\ge0$. Similarly, taking $(x,y)=(s,t)$ and $(x,y)=(t,s)$, condition~\eqref{eq:Ghat} with $\sigma=2$ therefore amounts to
\[
p+H(s)\ge0\ \Longrightarrow\ p+H(t)\ge0,
\qquad\qquad
p+H(t)\le0\ \Longrightarrow\ p+H(s)\le0 .
\]
Assume $p+H(s)\ge0$. If $H(s)>0$, then $p>0$ by {\eqref{eq:H-sign}}; if $H(s)=0$, then $p\ge0$. Either way $p\ge0$, and $H(t)\ge0$ gives $p+H(t)\ge0$. Now assume $p+H(t)\le0$. Then $p\le-H(t)\le0$, so {\eqref{eq:H-sign} forces} $H(s)=0$, whence $p+H(s)=p\le0$.

\emph{Step 2 ($h$ is $\mu$-quasiconvex for no $\mu>0$).}
Fix an integer $N\ge1$ and let $r_N:=N+1$. Then
\[
h'(r_N)=2r_N,
\qquad\qquad
h(r_N)=r_N^2+\sum_{n=1}^{N}n^4 .
\]
Let $\ell_N\le0$ be the point on the left branch with the same value. There $h(x)=x^2$, so
\[
-\ell_N=h(r_N)^{1/2}=\Bigl(r_N^2+\sum_{n=1}^{N}n^4\Bigr)^{1/2}\ge N^2 .
\]
The two equal-value points thus lie a distance $r_N-\ell_N\ge N^2$ apart, while the gradient at the right-hand one is only $2r_N=2(N+1)$.

\begin{figure}[H]
\centering
\begin{tikzpicture}[>=stealth,line join=round,
     dot/.style={circle,fill,inner sep=1.2pt},
     ref/.style={dashed,gray}]
  \draw[ultra thick,blue!55!black]
     plot[domain=0.6:6,samples=60] ({\x},{0.6+0.11*(6-\x)*(6-\x)});
  \draw[ultra thick,blue!55!black]
     (6,0.6)
     .. controls (6.35,0.62) and (6.55,0.72) .. (6.75,0.82)
     .. controls (6.84,0.88) and (6.92,1.30) .. (7.08,1.42)
     .. controls (7.30,1.52) and (7.55,1.58) .. (7.72,1.70)
     .. controls (7.80,1.78) and (7.90,2.35) .. (8.08,2.50)
     .. controls (8.20,2.60) and (8.28,2.66) .. (8.34,2.72)
     .. controls (8.40,2.80) and (8.47,3.54) .. (8.53,3.70)
     .. controls (8.62,3.733) and (8.71,3.767) .. (8.80,3.80);
  \draw[very thick,orange!85!black]
     (6.75,0.82) .. controls (6.84,0.88) and (6.92,1.30) .. (7.08,1.42);
  \draw[very thick,orange!85!black]
     (7.72,1.70) .. controls (7.80,1.78) and (7.90,2.35) .. (8.08,2.50);
  \draw[very thick,orange!85!black]
     (8.34,2.72) .. controls (8.40,2.80) and (8.47,3.54) .. (8.53,3.70);
  \draw[->,black] (-0.1,0) -- (10.9,0) node[right]{$x$};
  \draw[->,black] (6,0) -- (6,5.3) node[above]{$h(x)$};
  \draw[ref] (0.6,3.80) -- (8.80,3.80);
  \draw[ref] (6.75,0) -- (6.75,0.82);
  \draw[ref] (7.72,0) -- (7.72,1.70);
  \draw[ref] (8.53,0) -- (8.53,3.70);
  \node[left] at (5.85,3.80){$h(r_N)$};
  \node[dot] at (0.6,3.80){};
  \draw[thick,orange!85!black] (8.68,3.756) -- (9.68,4.123);
  \node[orange!85!black,above] at (9.35,4.25)
     {$h'(r_N)=2r_N=2(N+1)$};
  \node[dot,fill=black] at (8.80,3.80){};
  \foreach \xx/\lab in {0.6/$\ell_N$,6/$0$,7.72/$N$,8.80/$r_N=N+1$}{%
     \draw[black] (\xx,0.1) -- (\xx,-0.1);
     \node[below] at (\xx,-0.12){\lab};}
  \node at (6.85,-0.18){$\cdots$};
  \draw[decorate,decoration={brace,amplitude=4pt,mirror}]
     (0.6,-0.85) -- (8.80,-0.85)
     node[midway,below=2pt]{$r_N-\ell_N\ge N^2$};
\end{tikzpicture}
\caption{The distance between the equal-value points grows faster than the slope at $r_N$. The figure is schematic and not to scale.}
\label{fig:counterexample-gap}
\end{figure}

Suppose to the contrary that $h$ is $\mu$-quasiconvex for some $\mu>0$. Then $h$ satisfies $(B_\mu)$ by Theorem~\ref{thm:sqc}, and at the pair $(x,y)=(\ell_N,r_N)$ this reads
\[
2r_N(r_N-\ell_N)\ge\frac{\mu}{2}(r_N-\ell_N)^2,
\qquad\text{that is,}\qquad
\mu\le\frac{4r_N}{r_N-\ell_N}\le\frac{4(N+1)}{N^2}.
\]
Letting $N\to\infty$ gives $\mu\le0$, a contradiction.

\emph{Step 3 (an arbitrary modulus).}
Fix $\sigma>0$ and set $f:=(\sigma/2)h$. Both sides of \eqref{eq:Ghat} scale by $\sigma/2$, so $f$ is quasiconvex and satisfies \eqref{eq:Ghat} with parameter $\sigma$, hence $(C_\sigma)$. If $f$ is $\mu$-quasiconvex for some $\mu>0$, then $h=(2/\sigma)f$ is $(2\mu/\sigma)$-quasiconvex, contradicting Step~2. {Thus $f$ satisfies $(C_\sigma)$ yet is $\mu$-quasiconvex for no $\mu>0$.}
}%
\end{example}

\begin{remark}\label{rem:negative-answer}
Since its premise is weaker, $(\widehat C_\sigma)$ is a stronger requirement on $f$ than $(C_\sigma)$. Its failure to imply strong quasiconvexity therefore refutes sufficiency for every condition it implies, in particular for $(C_\sigma)$.{ This answers the converse question of \cite[Remark~9]{lara25}, the reverse implications left open after \cite[Proposition~10]{lara25}, and the claim in \cite[Remark~12]{lara24preprint}. Of the relations between $(C_\sigma)$ and $\sigma$-quasiconvexity, only the implication of Proposition~\ref{prop:necessity-C} holds.}
\end{remark}

\begin{remark}[Relation to strong pseudomonotonicity]\label{rem:pseudomonotone}
Let $\sigma>0$. Since $\la\nabla f(x),y-x\ra\ge0$ implies the premise of $(C_\sigma)$ for $x\ne y$, condition $(C_\sigma)$ implies that $\nabla f$ is strongly pseudomonotone with modulus $\sigma/2$ in the sense of \cite{karamardian90,hadjisavvas93}: for all $x,y\in C$,
\[
\la\nabla f(x),y-x\ra\ge0
\quad\Longrightarrow\quad
\la\nabla f(y),y-x\ra\ge\frac{\sigma}{2}\|y-x\|^2 .
\]
Example~\ref{ex:no-positive-modulus} therefore also shows that strong pseudomonotonicity of the gradient, with any modulus, does not imply strong quasiconvexity.
\end{remark}

For a G\^ateaux differentiable quasiconvex function $f$ on an open convex set, put
\[
\operatorname{sqc}(f):=\sup\{\sigma\ge0:\ f\text{ is }\sigma\text{-quasiconvex}\},
\qquad
\operatorname{qm}(f):=\sup\{\sigma\ge0:\ (C_\sigma)\text{ holds}\}.
\]
Both index sets are intervals containing $0$: if $(C_\sigma)$ holds and $0\le\sigma'<\sigma$, then $(C_{\sigma'})$ holds, because the premise of $(C_{\sigma'})$ implies that of $(C_\sigma)$ and the conclusion of $(C_\sigma)$ implies that of $(C_{\sigma'})$; the same monotonicity holds for $(A_\sigma)$.

\begin{proposition}[Comparison of moduli]\label{prop:moduli}
Let $C\subseteq X$ be a nonempty open convex set, and let $f\colon C\to\R$ be G\^ateaux differentiable and quasiconvex. Then $\operatorname{sqc}(f)\le\operatorname{qm}(f)$.
\end{proposition}

\begin{proof}
By Proposition~\ref{prop:necessity-C}, $(C_\sigma)$ holds whenever $f$ is $\sigma$-quasiconvex, so the first index set is contained in the second.
\end{proof}

\begin{remark}\label{rem:gap-maximal}
The inequality in Proposition~\ref{prop:moduli} can be strict, and the two moduli can be arbitrarily far apart. For the $h$ of Example~\ref{ex:no-positive-modulus}, Step~1 gives $(C_2)$ while Step~2 excludes every positive strong-quasiconvexity modulus. To see that $\operatorname{qm}(h)\le2$, fix $\tau>2$ and set $s:=-\tfrac12$ and $t:=\tfrac12$. Since $s,t\le1$, we have $H(s)=H(t)=0$, so $h'(s)=-1$ and $h'(t)=1$, and
\[
h'(t)(s-t)=h'(s)(t-s)=-1>-\frac{\tau}{2}.
\]
Hence the premise of $(C_\tau)$ holds at $(s,t)$ but its conclusion fails. Therefore
\[
\operatorname{qm}(h)=2>0=\operatorname{sqc}(h).
\]
For $a>0$, condition $(C_\rho)$ for $ah$ is equivalent to $(C_{\rho/a})$ for $h$, and the same scaling rule holds for strong quasiconvexity. Hence, for $f_\sigma=(\sigma/2)h$,
\[
\operatorname{qm}(f_\sigma)=\sigma,\qquad \operatorname{sqc}(f_\sigma)=0.
\]
\end{remark}

\begin{remark}\label{rem:smooth-counterexample}

The construction uses $H$ only through \eqref{eq:H-properties}. On each interval $[n,n+1]$, replace the polynomial $6n^4(t-n)(n+1-t)$ by
\[
H(t):=n^4\psi(t-n),
\]
where $\psi\colon\R\to\R$ is $C^\infty$, nonnegative, vanishes outside $(0,1)$, and satisfies $\int_0^1\psi=1$. For instance, one may take
\[
\psi(t):=
\begin{cases}
\kappa\exp\!\left(-\dfrac{1}{t(1-t)}\right),&0<t<1,\\[1mm]
0,&\text{otherwise},
\end{cases}
\]
where $\kappa>0$ normalizes the integral. This replacement preserves all four properties in \eqref{eq:H-properties}, including $H=0$ on $(-\infty,1]$, so Steps~1--3 and Remark~\ref{rem:gap-maximal} apply verbatim. Since $\psi$ vanishes to infinite order at $0$ and $1$, the cells glue smoothly and $h\in C^\infty(\R)$. Thus the counterexample may be taken of class $C^\infty$.

\end{remark}

}%

\section{\texorpdfstring{{Exact characterizations via compensated profiles}}{Exact characterizations via compensated profiles}}
{%
The compensated profile exposes the interior behavior along a segment that the endpoint condition $(C_\sigma)$ cannot see, yielding the following two exact characterizations of $\sigma$-quasiconvexity.

\begin{definition}[Compensated profile]\label{def:profile}
{Let $C\subseteq X$ be convex.} For a function $f\colon {C}\to\R$, a number $\sigma\ge0$, and distinct points $x,y\in {C}$, the \emph{compensated profile} along the segment $[x,y]$ is
\[
\varphi_{x,y}(t):=f(x+t(y-x))+\frac{\sigma}{2}t(1-t)\|y-x\|^2,\qquad t\in[0,1].
\]
\end{definition}

{Suppose in addition that $C$ is open and that $f$ is G\^ateaux differentiable, as in both theorems below. For fixed distinct $x,y\in C$, the defining formula then extends $\varphi_{x,y}$ to a differentiable function on an open interval containing $[0,1]$; in particular, $\varphi_{x,y}$ is continuous on $[0,1]$, and
\begin{equation}\label{eq:profile-derivative}
\varphi_{x,y}'(t)
=\la\nabla f(x+t(y-x)),y-x\ra
+\frac{\sigma}{2}(1-2t)\|y-x\|^2 .
\end{equation}
}%
Since $\varphi_{x,y}(0)=f(x)$ and $\varphi_{x,y}(1)=f(y)$, inequality~\eqref{eq:SQC} states exactly that $\varphi_{x,y}$ never exceeds its endpoint maximum. The first characterization detects a failure of this property through an interior sign change.

\begin{theorem}\label{thm:completed-characterization}
Let $\sigma\ge0$, let {$C\subseteq X$ be open and convex}, and let $f\colon {C}\to\R$ be {G\^ateaux differentiable}. Then $f$ is $\sigma$-quasiconvex if and only if, for all distinct $x,y\in {C}$ and all $0<s<t<1$,
\begin{equation*}\tag{$D_\sigma$}\label{eq:D-sigma}
\varphi_{x,y}'(s)>0>\varphi_{x,y}'(t)
\quad\Longrightarrow\quad
\min\{\varphi_{x,y}(s),\varphi_{x,y}(t)\}
\le
\max\{\varphi_{x,y}(0),\varphi_{x,y}(1)\}.
\end{equation*}
\end{theorem}

\begin{proof}
If $f$ is $\sigma$-quasiconvex, then
\[
\varphi_{x,y}(r)
\le\max\{\varphi_{x,y}(0),\varphi_{x,y}(1)\},
\qquad r\in[0,1].
\]
Hence $(D_\sigma)$ follows.

Conversely, suppose that $(D_\sigma)$ holds but $f$ is not $\sigma$-quasiconvex. {Then \eqref{eq:SQC} fails for some distinct $x,y\in C$. Write
\[
M:=\max\{\varphi_{x,y}(0),\varphi_{x,y}(1)\},
\qquad
S:=\bigl\{r\in(0,1):\varphi_{x,y}(r)>M\bigr\},
\]
so that $S$ is open and nonempty. Let $(\alpha,\beta)$ be one of its connected components, so that $0\le\alpha<\beta\le1$. Since $\varphi_{x,y}>M$ on $(\alpha,\beta)$, continuity gives $\varphi_{x,y}(\alpha)\ge M$ and $\varphi_{x,y}(\beta)\ge M$. The reverse inequalities hold as well. Indeed, if $\alpha=0$, then $\varphi_{x,y}(\alpha)=\varphi_{x,y}(0)\le M$; and if $\alpha\in(0,1)$, then $\alpha\notin S$ because $(\alpha,\beta)$ is a maximal subinterval of $S$, which again means $\varphi_{x,y}(\alpha)\le M$. The same holds at $\beta$. Hence
\[
\varphi_{x,y}(\alpha)=\varphi_{x,y}(\beta)=M .
\]
Choose $r\in(\alpha,\beta)$. The mean value theorem on $[\alpha,r]$ and $[r,\beta]$ gives points $s\in(\alpha,r)$ and $t\in(r,\beta)$ such that
\[
\varphi_{x,y}'(s)>0>\varphi_{x,y}'(t).
\]
Both $s$ and $t$ lie in $S$, so $\min\{\varphi_{x,y}(s),\varphi_{x,y}(t)\}>M$, contradicting $(D_\sigma)$.}
\end{proof}

\begin{remark}[Endpoint and interior tests]\label{rem:completed-characterization}
Under the assumptions of Theorem~\ref{thm:completed-characterization}, condition $(D_\sigma)$ admits the following equivalent formulation. For every distinct $x,y\in C$, write
\[
\varphi:=\varphi_{x,y},
\qquad
M:=\max\{\varphi(0),\varphi(1)\}.
\]
We require that
\[
\varphi'(s)>0
\qquad\Longrightarrow\qquad
\varphi'(t)\ge0
\]
for the following pairs:
\begin{enumerate}[(i)]
\item $(s,t)=(0,1)$;
\item $0<s<t<1$ with $\min\{\varphi(s),\varphi(t)\}>M$.
\end{enumerate}

The interior requirement (ii) is equivalent to $(D_\sigma)$ by contraposition. To show that $(D_\sigma)$ also implies the endpoint requirement (i), suppose that $(D_\sigma)$ holds but
\[
\varphi'(0)>0>\varphi'(1).
\]
Since $M$ is attained at $0$ or at $1$, the sign of $\varphi'$ at that endpoint gives $\varphi(u)>M$ for interior $u$ close to it. Applying the mean-value argument in the proof of Theorem~\ref{thm:completed-characterization} to this same profile yields a pair $0<s<t<1$ such that
\[
\varphi'(s)>0>\varphi'(t),
\qquad
\min\{\varphi(s),\varphi(t)\}>M,
\]
contradicting $(D_\sigma)$. Thus adding the endpoint requirement does not change the condition.

Finally, the derivative formula~\eqref{eq:profile-derivative} gives
\[
\begin{aligned}
\varphi'(0)
&=\la\nabla f(x),y-x\ra
+\frac{\sigma}{2}\|y-x\|^2,\\
\varphi'(1)
&=\la\nabla f(y),y-x\ra
-\frac{\sigma}{2}\|y-x\|^2.
\end{aligned}
\]
Hence requirement (i) is exactly $(C_\sigma)$. In this formulation, $(C_\sigma)$ is the restriction to the endpoint pair, while $(D_\sigma)$ includes the interior tests above the endpoint maximum.
\end{remark}

Remark~\ref{rem:completed-characterization} shows that $(D_\sigma)$ is the complete version of $(C_\sigma)$: it retains $(C_\sigma)$ as its endpoint clause (i) and adds the interior clause (ii). By Theorem~\ref{thm:completed-characterization}, this addition turns the necessary condition $(C_\sigma)$ into a characterization of $\sigma$-quasiconvexity, and by Example~\ref{ex:no-positive-modulus} it cannot be omitted for any $\sigma>0$. When $\sigma=0$, however, the interior clause is implied by the endpoint clause (Corollary~\ref{cor:quasimonotone}), and Remark~\ref{rem:sigma-zero} shows why this fails for $\sigma>0$.

\begin{corollary}[The case $\sigma=0$]\label{cor:quasimonotone}
Let $C\subseteq X$ be open and convex, and let $f\colon C\to\R$ be G\^ateaux differentiable. Then $(C_0)$ is equivalent to $(D_0)$.
\end{corollary}

\begin{proof}
Although this follows from Theorems~\ref{thm:qc} and~\ref{thm:completed-characterization}, the direct argument shows where the interior clause comes from. By Remark~\ref{rem:completed-characterization}, it suffices to show that $(C_0)$ implies requirement (ii). Fix distinct $x,y\in C$ and $0<s<t<1$, and put $x_r:=x+r(y-x)$. For $\sigma=0$, formula~\eqref{eq:profile-derivative} reads $\varphi_{x,y}'(r)=\la\nabla f(x_r),y-x\ra$. Since $x_t-x_s=(t-s)(y-x)$ with $t-s>0$, the inequality $\varphi_{x,y}'(s)>0$ means $\la\nabla f(x_s),x_t-x_s\ra>0$, and $(C_0)$ applied to the pair $(x_s,x_t)$ gives $\la\nabla f(x_t),x_s-x_t\ra\le0$, that is, $\varphi_{x,y}'(t)\ge0$. Thus (ii) holds, even without the restriction $\min\{\varphi(s),\varphi(t)\}>M$.
\end{proof}

\begin{remark}\label{rem:sigma-zero}
Condition $(C_0)$ is quasimonotonicity of $\nabla f$: for all $x,y\in C$,
\[
\la\nabla f(x),y-x\ra>0
\quad\Longrightarrow\quad
\la\nabla f(y),x-y\ra\le0 .
\]
Together with Theorem~\ref{thm:completed-characterization} for $\sigma=0$, Corollary~\ref{cor:quasimonotone} states that $f$ is quasiconvex if and only if $\nabla f$ is quasimonotone; this recovers the equivalence {\rm (a)}$\Leftrightarrow${\rm (c)} of Theorem~\ref{thm:qc} and the observation in \cite[Remark~9]{lara25}.

For $\sigma>0$ the same computation no longer closes. Applying $(C_\sigma)$ to the interior pair $(x_s,x_t)$, with $x_t-x_s=(t-s)(y-x)$, gives the implication
\[
\la\nabla f(x_s),y-x\ra>-\frac{\sigma}{2}(t-s)\|y-x\|^2
\quad\Longrightarrow\quad
\la\nabla f(x_t),y-x\ra\ge\frac{\sigma}{2}(t-s)\|y-x\|^2,
\]
whereas, by \eqref{eq:profile-derivative}, requirement (ii) asks for
\[
\la\nabla f(x_s),y-x\ra>-\frac{\sigma}{2}(1-2s)\|y-x\|^2
\quad\Longrightarrow\quad
\la\nabla f(x_t),y-x\ra\ge-\frac{\sigma}{2}(1-2t)\|y-x\|^2 .
\]
For $\sigma=0$ all thresholds vanish and the two implications coincide; for $\sigma>0$ they differ, because the compensated profile carries the whole segment $[x,y]$ in its quadratic term, while $(C_\sigma)$ at $(x_s,x_t)$ sees only the subsegment. Example~\ref{ex:no-positive-modulus} exploits exactly this mismatch.
\end{remark}

The following elementary criterion gives a second exact first-order characterization. It tests the strong quasiconvexity inequality only at critical points of the compensated profile and contains no integral.

\begin{proposition}\label{prop:critical-point-characterization}
Let $\sigma\ge0$, let $C\subseteq X$ be open and convex, and let $f\colon C\to\R$ be G\^ateaux differentiable. Then $f$ is $\sigma$-quasiconvex if and only if, {for all distinct $x,y\in C$ and all $t\in(0,1)$,
\begin{equation*}\tag{$E_\sigma$}\label{eq:E-sigma}
\varphi_{x,y}'(t)=0
\quad\Longrightarrow\quad
\varphi_{x,y}(t)\le
\max\{\varphi_{x,y}(0),\varphi_{x,y}(1)\}.
\end{equation*}}
\end{proposition}

\begin{proof}
{If $f$ is $\sigma$-quasiconvex, the conclusion of $(E_\sigma)$ holds for every $t\in[0,1]$ by \eqref{eq:SQC}.} Conversely, suppose that $(E_\sigma)$ holds but $f$ is not $\sigma$-quasiconvex. Then, for some $x,y$ and $t_0\in(0,1)$,
\[
\varphi_{x,y}(t_0)>
\max\{\varphi_{x,y}(0),\varphi_{x,y}(1)\}.
\]
Since $\varphi_{x,y}$ is continuous on $[0,1]$, it attains its maximum at some $\bar t\in[0,1]$, and $\bar t\in(0,1)$ because $\varphi_{x,y}(t_0)>\max\{\varphi_{x,y}(0),\varphi_{x,y}(1)\}$. At this interior point $\varphi_{x,y}'(\bar t)=0$, and $(E_\sigma)$ gives
\[
\varphi_{x,y}(\bar t)\le\max\{\varphi_{x,y}(0),\varphi_{x,y}(1)\}.
\]
This contradicts $\varphi_{x,y}(\bar t)\ge\varphi_{x,y}(t_0)>\max\{\varphi_{x,y}(0),\varphi_{x,y}(1)\}$.
\end{proof}

\begin{remark}\label{rem:geometric-P}
{Written out through $f$ and $\nabla f$, condition $(E_\sigma)$ takes the following geometric form.} Let $z\in(x,y)$ and set
\[
e:=\frac{y-x}{\|y-x\|}.
\]
Then {\eqref{eq:profile-derivative} shows that} $(E_\sigma)$ is equivalent to
\[
\la\nabla f(z),e\ra
=\frac{\sigma}{2}\bigl(\|z-x\|-\|y-z\|\bigr)
\quad\Longrightarrow\quad
f(z)\le\max\{f(x),f(y)\}
-\frac{\sigma}{2}\|z-x\|\,\|y-z\|.
\]
Thus one checks the strong quasiconvexity inequality only at the critical points of the compensated profile $\varphi_{x,y}$. This is a structural reduction, not a finite test, because the condition still ranges over all segments. When $\sigma=0$, condition $(E_0)$ says that an interior stationary point of a line restriction cannot lie above both endpoint values. By Proposition~\ref{prop:critical-point-characterization} this is equivalent to quasiconvexity.
\end{remark}
}%
\section{\texorpdfstring{{Conclusion}}{Conclusion}}
{This paper addresses three first-order questions about strong quasiconvexity: how the value-order condition tracks the modulus, whether the generalized quasimonotonicity condition is sufficient, and how to complete it into an exact characterization. {In the G\^ateaux differentiable setting, the value-order condition characterizes strong quasiconvexity with the same modulus.} In contrast, a $C^\infty$ counterexample shows that the generalized quasimonotonicity condition of Lara, Marcavillaca, and Vuong is not sufficient for any positive modulus, even in its stronger nonstrict form. This answers their open question negatively and quantifies the gap of Proposition~\ref{prop:moduli}: $\operatorname{qm}(f_\sigma)=\sigma$ while $\operatorname{sqc}(f_\sigma)=0$ (Remark~\ref{rem:gap-maximal}).

{For G\^ateaux differentiable functions, condition $(D_\sigma)$ admits an equivalent formulation whose endpoint clause is precisely the generalized quasimonotonicity condition and whose interior clause supplies the information needed for equivalence with strong quasiconvexity.} Condition $(E_\sigma)$ gives an independent characterization through the critical points of the same compensated profile. Both reduce to first-order descriptions of ordinary quasiconvexity when $\sigma=0$.}

A natural next problem is whether the exact modulus persists in the nonsmooth setting of \cite{hadjisavvaslara25}. There the line restriction need not be differentiable, so the inverse-function formula and the absolute continuity of the endpoint maps used in Steps~3--4 of Lemma~\ref{lem:one-dimensional-exact} are no longer available.

\section*{Acknowledgments}
The authors are grateful to Professor Nicolas Hadjisavvas for pointing out an error in the proof of an earlier version of this paper, and to Professor Felipe Lara for his insightful comments, which helped us improve the quality of the manuscript.

\bibliographystyle{mystyle}
{\small
\bibliography{References}
}
\end{document}